\documentclass[12 pt]{amsart}
\usepackage{amssymb, amsmath, amsfonts, amsthm, graphics}
\usepackage[hmargin=1 in, vmargin = 1 in]{geometry}
\usepackage{tikz-cd}
\usepackage{hyperref}
\usepackage[all]{xy}
\usepackage{enumitem}
\usepackage{youngtab} % This is the package I use to draw Young Tableaux. Check out http://www.ctex.org/documents/packages/math/youngtab.pdf for documentation.
\usepackage{mathtools}
\usepackage[utf8]{inputenc} % allow utf-8 input
\usepackage[T1]{fontenc}    % use 8-bit T1 fonts
\usepackage{hyperref}
\usepackage[capitalise]{cleveref}% hyperlinks
\usepackage{url}            % simple URL typesetting
\usepackage{booktabs}       % professional-quality tables
\usepackage{amsfonts}       % blackboard math symbols
\usepackage{nicefrac}       % compact symbols for 1/2, etc.
\usepackage{microtype}      % microtypography
\usepackage{lipsum}		% Can be removed after putting your text content
\usepackage{graphicx}
\usepackage{natbib}
\usepackage{doi}
\usepackage{comment}

\theoremstyle{definition}

\newtheorem*{example}{Example}

\newtheorem{thm}{Theorem}[section]
\newtheorem{cor}{Corollary}[thm]
\newtheorem{lem}{Lemma}[section]
\newtheorem{prop}{Proposition}[section]
\newtheorem{defn}{Definition}[section]

\newcommand{\Hom}{\text{Hom}}

\newcommand{\mycomment}[1]{}

\newcommand{\CC}{\mathbb{C}}

\newcommand{\RR}{\mathbb{R}}

\newcommand{\ZZ}{\mathbb{Z}}

\title{Residue formulae for the  trace on affine Hecke algebras.}

\author{Paul Mammen}
\address{Department of Mathematics, University of Michigan, 530 Church Street, Ann Arbor, MI
48109-1043, USA}
\email{\href{mailto:pmammen@umich.edu}{pmammen@umich.edu}}
\begin{document}

%%% Add PDF metadata to help others organize their library
%%% Once the PDF is generated, you can check the metadata with
%%% $ pdfinfo template.pdf

\maketitle
\begin{abstract}
	Motivated by recent advances in Catalan combinatorics, we study special values of the standard trace on affine Hecke algebras. Starting from a generating function for this trace calculated by Opdam, we use the theory of Szenes and Vergne to obtain residue formulae for the trace. This allows us to derive a product formula for the trace  of translation elements corresponding to weights in  certain ``Big Chambers'' of the positive root cone.
\end{abstract}
\setcounter{tocdepth}{1}
\tableofcontents

\section{Introduction}
\label{section1}
Given a Coxeter system $(W,S)$, one can associate to it a Hecke Algebra, $H_q(W,S)$ over $\ZZ[q,q^{-1}]$, where $q$ is an indeterminate. It has a natural basis $\{T_w\}_{w \in W}$ and we  can define the standard trace $$\tau: H_{q}(W,S) \to \ZZ[q,q^{-1}],\text{ } \tau(\sum_{w \in W}a_{w}T_{w}) = a_{e}.$$ Recent work (\cite{GLMW}, \cite{GL}) crucially uses the trace of Coxeter elements  in the Hecke algebra associated to in a \textit{finite} Coxeter group to derive new combinatorial formulae. In analogy with the finite case, Williams(\cite{wil}) proposed a framework in the case of the \textit{affine} Coxeter group, based on the trace of the associated (extended) affine Hecke algebra.

 In this paper, we build on the approach outlined in \cite{wil} and consider the trace of certain elements  in the affine Hecke algebra, associated to translation elements in the corresponding affine Weyl group. The trace on affine Hecke algebras have been intensively studied in relation to the representation theory of p-adic groups, starting with the work of Macdonald(\cite{Mac}), Opdam-Heckman(\cite{HO}) and in full generality by  Opdam(\cite{op1,op2}).  In \cite{op1}, Opdam found an interesting generating function for the trace of translation elements,  which is the starting point for the framework in \cite{wil}. Translation elements correspond to weights in the associated \textit{weight lattice}, and the generating function expresses the trace as a weighted sum over \textit{Kostant partitions}.

In the affine type $A$ case, Galashin, Lam and Williams (\cite{wil}) observed that the weighted sum from Opdam's formula matches a specialization of an entirely different weighted sum, first considered by Haglund (\cite{Ha}, \cite{AGH+}), over \textit{Tesler Matrices} with fixed hook sum $(b_1,\ldots,b_{n-1})$. The weighted sum of Tesler matrices with hook sum $(1,m,\ldots,m)$ can be interpreted as the Hilbert series of the $S_{n-1}$ module $DH_{n-1}^{m}$, which is a generalization of the ring of diagonal coinvariants. Haglund (\cite{Ha}) also showed that the specialization of the weighted sum for $(1,m,\ldots,m)$ had a particularly nice product formula. Building on the framework in \cite{wil},  Williams et. al. (\cite{EHKM+}) used the product formula to provide a new proof of Cayley's theorem on the number of trees in the complete graph. In \cite{AGH+}, a generalization of the product formula was conjectured for a weighted sum over Tesler matrices with positive hook sums $(b_1,\ldots,b_{n-1})$.

The identification between Kostant partitions and Tesler matrices follows from the reinterpretation of Tesler matrices as integer flows on the complete graph $K_{n}$ (\cite{MMR}). Building on this interpretation, Liu, Mészáros and  Morales (\cite{LMM}) proved  \cite[Conjecture 7.2]{AGH+}, in the context of $q$-Erhart sums for threshold graphs.

In this paper, we prove the following theorem, which is a generalization of \cite[ Conjecture 5.10]{wil} (The conjecture in \cite{wil} is a restatement of  \cite[Conjecture 7.2]{AGH+} in terms of the standard trace).
\begin{thm}
\label{thm1.1}

 Let $ \tilde{H} = H(\tilde{A}_{n-1})$ be the extended affine Hecke algebra.  Let  $\lambda = \sum_{i=1}^{n-1} a_i \alpha_{i}$, where 
 $\alpha_{i} = e_{i} - e_{i+1} $ denote the simple roots. Suppose 
         $$\label{eq1.1}a_1 > a_2 > \cdots > a_{n-m} < a_{n-m+1} \cdots < a_{n-1}.$$
    Then \[\tau(\theta_{-\lambda}) = q^{-\sum_{i=1}^{n-1}a_{i}}(q-1)^{2(n-1)}[a_{n-m}]_{q^n}\prod_{i=1}^{n-m-1}[(i+1)a_{i} - ia_{i+1}]_q\prod_{j=1}^{m-1}[(j+1)a_{n-j}-ja_{n-j-1}]_q,\]
    where $\theta_{-\lambda}$ is the basis element corresponding to the translation $t_\lambda \in \tilde{A}_{n-1}$ (\cref{def2.2}) and $[k]_{q} := \frac{q^{k}-1}{q-1}$ is the $k^{\text{th}}$ $q$-integer.
    
\end{thm}

Our approach to proving \cref{thm1.1} is to analyze the trace generating function using the theory of Euler-MacLaurin sums, due to Szenes and Vergne (\cite{sv}).  The approach in \cite{sv} uses a multidimensional generalization of residues, the Jeffery-Kirwan residue, and allows us to write a  formula for coefficients that belong to a \textit{Big Chamber} of the cone generated by positive roots. Given a big chamber, we can write the coefficients as a sum over certain 
``residual points''.

The paper is structured as follows. We begin by summarizing the necessary definitions and results on the trace of the affine Hecke algebra from \cite{op1} and \cite{op2} in \cref{section2}.
In \cref{section3}, we will explain the approach outlined in \cite{sv}  with examples. The main theorem in \cite{sv}, expresses the coefficients of a generating function as sums of residues at certain poles.  In \cref{section4}, we put it all together and analyze Opdam's generating function. In type A, We show that the inequalities in \cref{thm1.1} define a big chamber. The poles that contribute to the sum of the residues are indexed by permutations, and we determine the permutations that contribute to the residue of the elements in  \cref{thm1.1}. We describe the permutations inductively and show that they correspond to products of the longest elements in the Weyl group.

\subsection*{Acknowledgements}
I thank my advisor, Thomas Lam, for his encouragement and comments. I thank Dawei Shen, Mia Smith, Amanda Schwartz and Calvin Yost-Wolff for useful discussions.

\section{Preliminaries on Affine Hecke Algebras}
\label{section2}
We will introduce the extended affine Hecke algebra, following \cite{op1}.
Let $\Phi \subset V^*$ be an irreducible root system of a Weyl group $W(\Phi) = W$, where we assume that $\Phi$ spans $V^*$. We have the associated set of co-roots $\Phi^{\vee} \subset V$, with a  bijection $\alpha \to \alpha^{\vee}$, where $\langle \alpha, \alpha^{\vee}\rangle = 2$, and $\langle \cdot, \cdot \rangle$ denotes the natural pairing between $V^*$ and $V$.

We will fix a simple system $\Delta \subset \Phi$, and define a partial order on $V^*$, where $\lambda > \mu \iff \lambda - \mu \in \bigoplus_{\alpha \in \Delta} \RR^+\alpha $. With respect to this partial order, we have a decomposition into $\Phi = \Phi^+ \cup \Phi^-$, where $\Phi^{\pm} := \{ \alpha \in \Phi | \pm \alpha > 0\} $.  With respect to this  partial order, we also have the longest root $\alpha_{0} \in \Phi^+$. The Weyl group $W$ is generated by the collection $S$ of simple reflections $s_\alpha$, associated to each $\alpha \in \Delta$, which  act on $V^*$ by \[s_{\alpha}(x) = x -\langle x,\alpha^{\vee}\rangle \alpha.\]
The rank of $\Phi$ is defined to be the cardinality of $\Delta$. It is also equal to the dimension of the vector space spanned by $\Phi$,  which in this case is the dimension of  $V^*$.

We have the coroot lattice $Q^{\vee} := \ZZ \Phi^{\vee} \subset V$ and the associated weight lattice \[P := \{\lambda \in V^*| \langle \lambda,\alpha^\vee \rangle \in \ZZ \text{ for all } \alpha^\vee \in \Phi^{\vee} \},\] and a perfect pairing between the two inherited from the pairing between $V$ and $V^{*}$. Inside $P$, we also have the root lattice $Q := \ZZ \Phi$. Both $P$ and $Q$ are preserved by $W$.
The weight lattice $P$ acts on $V^*$ by translation and for an element $\lambda \in P$, we will denote the associated translation by $t_{\lambda}$.  We define the extended affine Weyl group $\tilde{W} :=W \ltimes
P$ inside the group of affine linear transformations on $V^*$. We also define the affine Weyl group $W_{\textit{aff}} := W \ltimes Q$. The group $W_{\textit{aff}}$ is an affine Coxeter group and is generated by the generators $s_\alpha \in S$ of the finite Weyl group $W$, along with $s_0 = s_{\alpha_0}t_{-\alpha_0}$. We will call this set of generators $S_{0}$.  

\begin{defn}
 For an element $x = wt_{\lambda} \in W \ltimes P $ we define its length \[\ell(wt_{\lambda}) := \sum_{\alpha \in \Phi^+} |\langle \lambda, \alpha^{\vee} \rangle + \chi_{>0}(w(\alpha))|, \]  where, $\chi_{>0}(v) = 1$ if $v > 0 $ in the partial order on $V^*$ and $0$ otherwise.   
\end{defn}
 
\begin{defn}
\label{def2.2}
    The extended affine Hecke Algebra $H(\tilde{W})$ is the algebra over $\ZZ[q,q^{
    -1}]$ generated by elements $T_x$, $x \in \tilde{W}$, subject to the relations 
    \begin{enumerate}
        \item $T_{x}T_{x'} = T_{xx'} \text{ if } \ell(xx') = \ell(x) + \ell(x')$, and
        \item $T_{s}^2  = (q-1)T_s + 1  \text{ for } s \in S_0 $.
    \end{enumerate}
\end{defn}

\begin{defn}
    The trace form $\tau: H \to \ZZ[q,q^{-1}]$ is the $\ZZ[q,q^{-1}]$-linear map,  where $\tau(T_w) = 1$, if $w $ is the identity
and $0$ otherwise.
\end{defn}

We have another presentation of the extended affine Hecke algebra $H$ with generators  $T_s$,  $s \in S$  and $\theta_ {\lambda}$ for $\lambda \in P$. This presentation is also popularly called the Bernstein presentation.
Before we state this description, we define the set of dominant weights
\[P^+ := \{ \lambda \in P | \langle\lambda, \alpha^{\vee} \rangle > 0, \text{ for all } \alpha \in \Delta\} \subset P. \]
We define the positive root cone $Q^+= \{v |  v = \sum_{\alpha \in \Delta} m_{\alpha}\alpha, \text{} m_{\alpha} \in \ZZ_{\geq 0}, \textrm{ for all } \alpha \in \Delta\}\subset Q$.

\begin{defn}\label{def2.2}
    For $\lambda \in P$, if $\lambda = \lambda_+ - \lambda_{-}$, where $\lambda_{\pm} \in P^+$,  we  define \[\theta_{\lambda} :=  q^{\ell(t_{\lambda_{-}})-\ell(t_{\lambda_{+}})/2}T_{\lambda_+}T_{\lambda_-}^{-1}.\]
 $\theta_\lambda$ does not depend on the decomposition of $\lambda$ as a difference of dominant weights.    
\end{defn}

\begin{thm}
       The algebra $H$ is generated by elements  $T_s$, $s \in S$  and $\theta_ {\lambda}$, $\lambda \in P$ subject to the relations
       \begin{enumerate}
           \item $T_{s}^2  = (q-1)T_s + 1  \text{ for } s \in S $,
           \item $\theta_{\lambda}\theta_{\mu} = \theta_{\lambda+\mu}$ for $\lambda,\mu \in P$,
           \item $T_s\theta_{\lambda}- \theta_{\lambda}T_s = (q-1)\frac{\theta_{\lambda}-\theta_{s(\lambda)}}{1-\theta_{-\alpha}}$.
       \end{enumerate}
    Thus, $A := \ZZ[q,q^{-1}](\theta_{\lambda})_{\lambda \in P}$ is  a commutative subalgebra of $H$ isomorphic to the group algebra of $P$.   
\end{thm}

\subsection{Trace formula and residual points}
We will assume $q > 1$ is a real number and set $k := \log(q)$.

\begin{defn}
     Let $T := \Hom(P, \CC^{*})$.  Define the rational function $\eta(t)$ on $T$ by  \[ \eta(t) := 
     %\prod_{\alpha \in \Phi^{+}} \frac{(1-t(\alpha))(1-t(-\alpha))}{q(1-q^{-1}t(\alpha))(1-q^{-1}t(-\alpha))} = %
     \prod_{\alpha \in \Phi^{+}} \frac{(1-t(\alpha))^2}{(1-q^{-1}t(\alpha))(1-qt(\alpha))},\]
     where $t \in T$.
\end{defn}

 We will be interested in a power series expansion of $\eta(t)$, which we will denote by $r^+(\eta(t))$. If we define $r^+(\frac{1}{1-x}) := \sum_{i=0}^\infty x^i$ and extend multiplicatively,  we can expand $\eta(t)$  as a power series in the $t_i = t(\alpha_i)$, where $\alpha_{1},\cdots,\alpha_{n}$ are the simple roots in $\Delta$,   
 \begin{equation*}
     \begin{split}
         r^+(\eta(t)) &= \prod_{\alpha \in \Phi^+}[(1-t(\alpha))^2(\sum_{i=0}^{\infty} q^it(\alpha)^{i})(\sum_{j=0}^{\infty}q^{-j}t(\alpha)^j)].
     \end{split}
 \end{equation*}

\begin{thm}[{\cite[Theorem 1.14]{op1}}]
\label{thm2.2}\

    We have the following equality of formal Laurent series \[\sum_{\lambda \in P} \tau(\theta_{-\lambda})t(\lambda) = r^+(\eta(t)).\]

In particular, if $\lambda$ is not in $Q^+$, then $\tau(\theta_{\lambda})$ vanishes.
\end{thm}

\begin{example}
    We can use  \cref{thm2.2} to  verify \cref{thm1.1} for $n = 2$.  In this case, the root system is $\Phi = A_1$.  We have a single positive root $\Phi^+ = \{\alpha_1\} = \Delta$.  We compute
          \begin{equation*}
         \begin{split}
         r^+(\eta(t)) &= \frac{(1-t(\alpha_1))^2}{(1-q^{-1}t(\alpha_1))(1-qt(\alpha_1))} \\
         &= (1 - 2t_1 + t_1^2)(\sum_{i=0} q^{i}t_1^i)(\sum_{j=0} q^{-j}t_1^j)\\
         &= (1 - 2t_1 + t_1^2)(\sum_{r=0}(\sum_{i=0}^{r}q^{2i-r})t_1^r)\\
         &= (1 - 2t_1 + t_1^2)(\sum_{r=0}\frac{q^{r+2}-q^{-r}}{(q^2-1)}t_1^r)\\
         &= \sum_{r=0}\frac{q^{r+2}-q^{-r} - 2(q^{r+1}-q^{-r+1})+ q^{r}-q^{-r+2}}{q^2-1}t_1^r\\
         &= \sum_{r=0}q^{-r}(q-1)^2\frac{q^{2r}-1}{q^2-1}t_1^r.
         \end{split}
     \end{equation*}  
Therefore, we have,    
      \[ \tau(\theta_{-r\alpha}) = q^{-r}(q-1)^2[r]_{q^2}. \]
      We define $\textrm{wt}(r) :=  q^{-r}(q-1)^2[r]_{q^2}$.
\end{example}

\begin{defn}
\label{def2.6}
For $\lambda \in Q^{+}$, a \textit{Kostant partition} is a tuple of non-negative integers $(r_{\alpha})_{\alpha \in \Phi^{+}}$ such that $\sum_{\alpha \in \Phi^{+}} r_{\alpha}\alpha = \lambda$.  The set of all such Kostant partitions will be denoted by $K(\lambda)$.
    
\end{defn}

We have the following Corollary of  \cref{thm2.2}.
\begin{cor}
\label{cor2.1.1}
    For $\lambda  \in P$, $\tau(\theta_{-\lambda}) = 0$, unless $ \lambda \in Q^{+}$, in which case \[
    \tau(\theta_{-\lambda}) = \sum_{(r_{\alpha}) \in K(\lambda)} \prod_{\alpha \in \Phi^{+}}\textrm{wt}(r_{\alpha}).
    \]
\end{cor}
Following \cite{op2}, we define the concept of a residual point. Recall that we have fixed $q>1$. 
 \begin{defn}
     $t \in T$ is a \emph{residual point} if $ \# \{ \alpha \in \Phi| t(\alpha) = 1\} -  \# \{\alpha \in \Phi | t(\alpha) = q \} \geq \text{rank}(\Phi)$ 
\end{defn}

We can write $\CC^* = \RR^+S^1$, where $\RR^+$ is  the multiplicative group of positive real numbers and $S_1$ is the unit circle,  via the polar decomposition $\gamma = re^{i\theta}$. This allows us to decompose $T$ into a product of groups.
$T_u = \Hom(P, S^1)$ and $T_{rs} = \Hom(P,\RR^+)$.  Any element $t \in T$ can be written as $e^{m + in}$ where $m,n \in  \RR \otimes Q^{\vee}$. Given $t \in T$, any $p \in \CC \otimes Q^{\vee}$ such that $t_p = e^{p} = t$, is in a unique class $[p]$ modulo $2\pi i Q^{\vee}$.

Note that since $\eta$  is $W$ invariant, all the points in the $W$-orbit of a residual point will also be a residual point. 

 Let $G$ be the semisimple group over the complex numbers associated to the root datum $(P,Q^{\vee},\langle \cdot , \cdot \rangle, \Phi,\Phi^\vee)$. It is the simply connected complex Lie group of the same Dynkin type as $\Phi$, and $T$ can be identified with the maximal torus of this Lie group.  We have the associated Lie algebra $\mathfrak{g} = {\rm Lie}(G)$.

We have the following theorem.
\begin{thm}\label{thm2.3}\cite[Proposition B.2]{op2}
    Let $T = T_{rs}T_u$. Then each $W$-orbit of residual points contains a point $t = cs$, with $s \in T_u$ and $c \in T_{rs}$, such that
    \begin{enumerate}
        \item $\Phi_{s} = \{\alpha| s(\alpha) = 1\}$ is a root subsystem of the same rank as $\Phi$. For such points $s \in T \subset G$, the
centralizer $C_{\frak{g}}(s)$ of $s$ in $\mathfrak{g}$ is a semi-simple Lie algebra and $\Phi_s$ is the root system of $C_{\frak{g}}(s)$.  The points $s$ are in $1-1$ correspondence with vertices of the associated extended Dynkin diagram and $\Phi_s$ is the root system of the diagram obtained by deleting the corresponding vertex from the extended Dynkin diagram.
        \item $c \in T_{rs} = \Hom(P,\RR^+)$ is of the form $c = \exp(\frac{kh}{2})$,  where $h \in Q^\vee$ is an element such that the Dynkin diagram of $\Phi_s$ labeled by $\{\alpha(h)\}_{\alpha \in \Delta}$ is a weighted Dynkin diagram (Cf. \cite{car})  of $\Phi_{s}$ corresponding to a distinguished nilpotent orbit of $C_{\frak{g}}(s)$. 

Additionally, $t = cs$ , where $s$ and $c$ satisfy the above properties, is a residual point.
    \end{enumerate}

\subsection{Link with Tesler Matrices}.
We briefly discuss the link with Tesler Matrices. Let us specialize to the case of the root system $\Phi_{n} = A_{n-1} \subset V^*$, where $V^* \subset \CC^n$ is the subspace of all vectors whose coordinates add up to zero. We choose $\Phi^{+}_n = \{\alpha_{i,j} = e_j - e_i | 1 \leq j <i \leq n\}$. The simple roots are $\alpha_{i} = \alpha_{i,i+1}$ for $1 \leq i <n$. The root lattice $Q(A_{n-1})$ is the set of all vectors with integer entries adding up to zero.  
\begin{defn}
     Given a matrix $M = [m_{ij}]$, the $i
^{th}$ hook sum $b_i$ is $$b_i := \sum_{j\geq i} m_{ij} - \sum_{j < i}m_{ji}.$$ 
To an  $r\times r$ matrix, we can associate a hook sum vector $(b_1,\ldots,b_{r})$.
 \end{defn} Let $U_{r}$ denote the set of $r\times r$ upper triangular matrices with non-negative integer entries.

\begin{defn}
    Fix a vector $(b_{1},\ldots,b_{n-1}) \in \ZZ^{n-1}$.
       \[\mathcal{T}(b_{1},\cdots,b_{n-1}) := \{M | M \in U_{n-1}\text{, with hook sum vector } (b_{1},\ldots,b_{n-1})\}.\]
       An element of $\mathcal{T}(b_{1},\cdots,b_{n-1})$ is called a Tesler Matrix with hook sum $(b_{1},\cdots,b_{n-1})$.
\end{defn}

\begin{defn}
    For variables $q,t$ we define the weight function $\textrm{wt}_{q,t}$ as follows.

    For $m \in \ZZ_{> 0}, \textrm{wt}_{q,t}(m) = -(1-q)(1-t)\frac{q^m-t^m}{q-t}$ and $\textrm{wt}_{q,t}(0) = 1$.

    For a $r \times r$  matrix $M = [m_{ij}]$, $\textrm{wt}_{q,t}(M) =  \prod_{1 \leq i,j \leq r} \textrm{wt}_{q,t}(m_{ij})$.

\end{defn}

In \cite{Ha} and \cite{AGH+}, of special interest is the sum\[
P_{(b_1,\ldots,b_{n-1})}(q,t) = \frac{\sum_{M \in T(b_{1}\ldots,b_{n-1})} \textrm{wt}_{q,t}(M)}{(-(1-q)(1-t))^{n-1}}.
\]

Given  $(b_{1},\ldots,b_{n-1}) \in \ZZ^{n-1}$, we can associate to it  
$$ \lambda(b_{1},\cdots,b_{n-1}) = \sum_{i =1}^{n-1} b_{i} \alpha_{i,n} \in Q(A_{n-1}). $$
Consequently, a Tesler matrix $M \in T(b_{1},\ldots,b_{n-1})$ can be associated with the Kostant partition $(r_{\alpha_{i,j}}) \in K(\lambda(b_{1},\cdots,b_{n-1}))$, where $r_{\alpha_{i,j}} = m_{ij}$, if $i \neq j$ and $r_{\alpha_{i,n}} = m_{i,i}$. In fact, the map is a bijection between the two sets giving a natural translation between the two.  

In  \cite{MMR} and \cite{LMM}, the objects of interest are flows on the complete graph $K_n$ with edge set $E(K_n) = \{(i,j)| 1 \leq i,j \leq n\}$. A flow with net  flow vector $(b_{1},\cdots,b_{n})$ is a function $F:E(K_n) \to \RR^+$, such that 
$\sum_{j >i}F((i,j)) - \sum_{j < i}F((i,j)) = b_{i}$ for $1 \leq i \leq n$.
The  number $b_i$ is called the net flow at vertex $i$. Note that $b_n$ must be $-b_{1}-b_{2},\cdots-b_{n-1}$ for a flow to exist and thus the netflow vector $\sum_{i=1}^n b_{i}e_{i} = \lambda(b_{1},\ldots,b_{n-1}).$

A Kostant partition $(r_{\alpha_{i,j}}) \in K(\lambda(b_1,\ldots,b_{n-1})) $ directly corresponds to a flow with integer values and netflow $\lambda(b_{1},\ldots,b_{n-1})$, where $r_{\alpha_{i,j}}$ is taken to be the flow on the edge join vertices $i$ and $j$. The weighted sum of flows  in \cite{LMM}, called the weighted Erhart sum,  is identical to the one for Tesler Matrices.

On putting $t = q^{-1}$ in $\textrm{wt}_{q,t}$, we end up with the function $\textrm{wt}$ that gives the weights for the Kostant partition in the trace formula from  \cref{cor2.1.1}.

Thus, \[P_{(b_1,\cdots,b_{n-1})}(q,q^{-1}) = \frac{\tau(-\lambda(b_{1},\cdots,b_{n-1}))}{((q-1)(q^{-1}-1))^{n-1}}.\]

For $\lambda(b_{1},\cdots,b_{n-1})$, when each Kostant partition has at least $n-1$ contributing positive roots, $P_{(b_1,\cdots,b_n)}$ will be a polynomial in $q,q^{-1}$. In \cite{AGH+}, they observe that if $b_i > 0$ for all $i$, $\lambda(b_{1},\cdots,b_{n-1})$ naturally satisfies this condition and these are the vectors for which formula for $P_{(b_1,\cdots,b_{n-1})}(q,q^{-1})$ is conjectured. This is also the setting in which the results of \cite{LMM} are stated.  Note that \[\lambda(b_{1},\cdots,b_{n-1}) = \sum_{i = 1} ^{n-1} b_{i}\alpha_{i,n} = \sum_{j = 1}^{n-1} (\sum_{i = 1}^{j} b_j) \alpha_{j}.\]

If $b_i > 0$ for all $i$, $\lambda(b_{1},\cdots,b_{n-1}) = \sum_{i =1}  ^{n-1} a_{i} \alpha_i$ satisfies $a_1 < a_2 \cdots < a_{n-1}$ and one can check that
 \cref{thm1.1}, exactly recovers \cite[ Conjecture 5.10]{wil}. The set  $\{\alpha_{i,n}\}_{i=1}^n$ is a basis of $Q(A_{n-1})$, but in general  $\lambda \in Q^{+}(A_{n-1})$ will not have positive coefficients when expanded in this basis. 
\end{thm}

\section{Preliminaries on Total Residues and Coefficient formulas}
\label{section3}

%We will assume that our space has a fixed volume form. We fix the standard topology on $V$ and $V^{*}$.For a set $S \subset V \text{ or } V^{*}$,  we will use $S^{\circ}$ to denote its interior. 

%We will also assume that $\Phi^+ \subset \Gamma^*$, where $\Gamma^*$ is the $\ZZ$ - dual lattice of a $\ZZ$ lattice $\Gamma$ in $V$. 
\subsection{Total Residue}
In this section, we follow \cite{sv}. Let $E$ be a complex vector space of dimension $n$. Let $\Phi^+ \subset E^*$ be a finite collection of linear functionals that span $E^*$. For each $\alpha \in \Phi^{+}$, we can define a hyperplane $H_{\alpha} = \{z : \alpha(z) = 0 \}$. Let $H$ be the collection of the defined hyperplanes. 
\begin{defn}
    \[ R_{H} = \{ f \text{ | } f \text{ is a rational function on $V$ with poles along } \cup_{\alpha \in \Phi^+} H_{\alpha}\}. \]
\end{defn}

A general element $f$ of  $R_{H}$ looks like $\frac{p(z)}{\Pi_{i=1}^r \alpha_i(z)} $, where $p(z)$ is a polynomial and $[\alpha_{1},\ldots,\alpha_{r}] $ is a sequence of elements in $\Phi^+$.

Let $B(\Phi^+)$ be the collection of subsets of $\Phi^+$ that form a basis of $E^{*}$. Functions of the form $f_{\sigma} = \frac{1}{\Pi_{\alpha \in \sigma} \alpha(z)}$, where $ \sigma  \in B(\Phi^+)$ will be called simple elements. Together, they span a subspace $S_H$ of $R_H$. 

For any $v$ in $E$, we have the directional derivative $\partial_v$ in the direction of $v$. The following decomposition holds (\cite[Theorem 1.1]{sv}, \cite[Proposition 7]{bv}): 
\[R_H =(\sum \partial_v R_H) \bigoplus S_H.\] 
 \begin{defn}
     $\rm{Tres}$ $: R_H \to S_H$ is the natural restriction map defined via the above decomposition.
 \end{defn}

Following \cite[pg. 3]{sv}, we can extend the definition of  $\rm{\textrm{Tres}}$ to the algebra $\hat{R}_H$ of formal meromorphic functions on $E$ near zero, with poles in $H$. Near zero, any such function can be written as $\frac{P(z)}{\Pi_{i=1}^r \alpha_i(z)} $,  where $P(z)$ is a formal power series and $[\alpha_{1},\ldots,\alpha_{r}] $ is a sequence of elements in $\Phi^+$.  In that case, we define $ \textrm{Tres}(\frac{P(z)}{\Pi_{i=1}^r \alpha_i(z))}) = \textrm{Tres}(\frac{P_{r-n}(z)}{\Pi_{i=1}^r \alpha_i(z))}) $, where $P_{r-n}(z)$ is the  homogeneous component of degree $r-n$ in $P(z)$. 

\begin{example}\cite[Example 2]{sv}
   We will calculate the total residue of the following function: \[g(z_1,z_2) = \frac{e^{z_1}}{(1-e^{-z_1})(1-e^{-z_2})(1-e^{-z_1 - z_2})}.\]

    This function is in $\hat{R}_H$, where $\Phi^+ = \{e_1^{*},e_2^{*},e_1^{*} + e_2^{*} \}$ and $H$  is the corresponding hyperplane arrangement. 
    We can write 
     \[g(z_1,z_2) = \frac{1}{z_1z_2(z_1 + z_2)}(\frac{e^{z_1}z_1z_2(z_1+z_2)}{(1-e^{-z_1})(1-e^{-z_2})(1-e^{-z_1 - z_2})}) = \frac{P(z)}{z_1z_2(z_1 + z_2)}.\]

     So to find the total residue, we need to find the degree $3-2 = 1$ terms in $P(z_1,z_2)$. Note that $\frac{z}{1-e^{-z}} = 1 + \frac{z}{2} + \cdots$ close to zero and thus $P_1(z_1,z_2) = \frac{3z_1}{2} + \frac{z_2}{2} + \frac{(z_1+z_2)}{2}$.

     \[\textrm{Tres}(g(z_1,z_2)) = \textrm{Tres}(\frac{\frac{3z_1}{2} + \frac{z_2}{2} + \frac{(z_1+z_2)}{2}}{z_1z_2(z_1 + z_2)}) = \frac{3}{2z_2(z_1 + z_2)} + \frac{1}{z_2(z_1 + z_2)} + \frac{1}{z_1(z_1 + z_2)}\]
\end{example}

We have the following lemma,

\begin{lem}(\cite[Lemma 1.3]{sv})
\label{lem2.1}
    Consider \[g(z) = \frac{e^{\alpha(z)}}{\Pi_{i=1}^r(1-u_ie^{-\alpha_{i}(z)})},\]
    where $[\alpha_{1}, \ldots,\alpha_{r}]$ is a sequence of elements in $\Phi^+$ and $u_i \in \CC^*$ for all $i$.
    \begin{enumerate}
        \item If the collection of  $\alpha_i$ for which $u_i = 1$ fails to span $E^{*}$, then  $\textrm{Tres}(g(z)) = 0 $ .
        \item If the $\alpha_i$ for which  $u_i = 1$ form a basis of $E^{*}$, then 
        \[\textrm{Tres}(g(z)) =  \frac{1}{\Pi_{u_j=1}\alpha_{j}(z)}[\frac{(\Pi_{u_j=1}\alpha_{j}(z))e^{\alpha(z)}}{\Pi_{i=1}^r(1-u_ie^{-\alpha_{i}(z)})}]_{0} = \frac{1}{\Pi_{u_j=1}\alpha_{j}(z)} \frac{1}{\Pi_{u_i \neq 1}(1-u_i)}. \]
    \end{enumerate}
\end{lem}
\subsection{The ring $M^{\Phi^+}$ and Reduced Set of Poles}
%We will also assume that $\Phi^+ \subset \Gamma^*$, where $\Gamma^*$ is the $\ZZ$ - dual lattice of a $\ZZ$ lattice $\Gamma$ in $V$. 

Let $V$ be a real vector space of dimension $n$. We will fix $\ZZ$ modules  $\Gamma,\Gamma^{*}$ inside $V$ and $V^{*}$ respectively, dual to each other with respect to the standard pairing between  $V$ and $V^*$.  Let $vol$ denote the volume form that assigns volume $1$ to the parallelepiped spanned by a $\ZZ$ basis of $\Gamma^{*}$.
For subsets $A, B $ of $V$, we will use $A+B$ to denote their Minkowski sum, 
$$A + B := \{a + b \text{ | } a \in A, b \in B\}.$$
For a collection of vectors $\sigma = [\alpha_{1}, \ldots, \alpha_r]$, we will use $\textrm{Box}(\sigma)$ to denote the set  $\sum_{i = 1}^r [0,1]\alpha_r$.

Let $\Phi^+ \subset V^*$ be a finite collection of linear functionals that span $V^*$. We will assume $\Phi^{+} \subset \Gamma^{*}$ and that  they lie in a half space $V_{+}^*$ of $V^*$ .  We also fix the standard topology on $V$ and $V^{*}$. For a set $E$, we will denote its interior by $E^{\circ}$. We also have the complex vector space  $V \otimes \CC := V_{\CC} $. For each $\alpha \in \Phi^{+}$, we can define a hyperplane $H_{\alpha} = \{z  \in V_{\CC}: \alpha(z) = 0 \}$ . Let $H_{\Phi^{+}}$ be the collection of the defined hyperplanes. 

\begin{defn}
    $M^{\Phi^+}$ is the ring of functions on $V \otimes_{\RR} \CC := V_{\CC} $ generated by the functions $z \to e^{\alpha(z)}$ , where $\alpha$ in $\Gamma^*$ and $\frac{1}{1-ue^{\alpha_{i}(z)}}$ where $\alpha_{i} \in \Phi^+$, $u_i \in \CC^{*}$. 
    A general element $F$ in $M^{\Phi^+}$ looks like
    \[ F = \frac{\sum_{\zeta \in \Gamma^*} c_{\zeta}e^{\zeta}}{\prod_{i=1}^r(1-u_{i}e^{\alpha_i})} ,\]
    where $c_{\zeta},u_i \in \CC^*$,  $c_{\zeta} \neq 0$ for a finite number of $\zeta \in \Gamma^* $ and $\alpha_i \in \Phi^{+}$.
    We define
    \[
      N(F) := \{ \mu | \mu + \zeta \in \sum_{i=1}^r [0,1]\alpha_i,  \forall c_{\zeta} \neq 0 \}
    \]

    If we expand $\frac{1}{(1-u_{i}e^{\alpha_i})} $ as $\sum_{k=0}^{\infty} u_i^ke^{k\alpha_i}$, we can define a power series expansion $r^+(F)$ for any $F$ in $M^{\Phi^+}$, where \[r^+(F) :=(\sum_{\zeta \in \Gamma^*} c_{\zeta}e^{\zeta})(\prod_{i=1}^r(\sum_{k=0} u_i^ke^{k\alpha_i}) =  \sum_{\zeta \in \Gamma^*}c_{\zeta}e^{\zeta}\prod_{i=1}^r(\sum_{k=0} u_i^ke^{k\alpha_i})=\sum_{\zeta \in \Gamma^*} a_\zeta e^\zeta\]
\begin{defn}
    For $\zeta \in \Gamma^*$ and $F \in M^{\Phi^+}$, if  $r^+(F) = \sum_{\zeta \in \Gamma^*} a_{\zeta}e^{\zeta}$, then we define $c(r^+(F), \zeta) := a_\zeta$.
\end{defn}
\end{defn}

In general, if the set of $\alpha_i$ where $u_i = 1$ does not span $V^*_{\CC}$, the total residue of a function $F= \frac{\sum_{\zeta \in \Gamma} c_{\zeta}e^{\zeta}}{\prod_{i=1}^n(1-u_{i}e^{\alpha_i})}$ will vanish (\cref{lem2.1}). Suppose that we have a point $p$ such that the set of $\alpha_{i}$ for which $e^{\alpha_{i}(p)} = u_i^{-1}$ spans $V^*_{\CC}$.  The function $z \to F(p-z)$ may have non-vanishing total residue. Note that if we add an element of $2i\pi\Gamma$ to $p$ that will not change $e^{\alpha_{i}(p)}$ and so we consider  the collection of such points $p$ upto translation by $2i\pi\Gamma$ and call this set $RSP(F)$ or a reduced set of poles. \\

For $\lambda \in \Gamma^*$,  define \[s[F, \Gamma](\lambda) := \sum_{p \in RSP(F)} \textrm{Tres}(e^{-\lambda(p-z)}F(p-z)) \in S_{H_{\Phi^{+}}}.\]
This sum will be important in determining our series expansion later. 

\subsection{The cone $C(\Phi^+)$ and the main theorem }

For a collection of elements $\Delta\subset \Phi^+$, we can look at the positive cone generated by $\Delta$,  $$C(\Delta) := \sum_{\alpha \in \Delta} \RR^{+}\alpha = \{v |  v = \sum_{\alpha \in \Delta} m_{\alpha}\alpha, \text{} m_{\alpha} \in \RR_{\geq 0}\text{ for all } \alpha \in \Delta\}.$$

If $z \in C(\Phi^+) $ is in a cone generated by fewer than $n$ elements in $\Phi^+$,  we say that it is singular. The set of singular vectors is denoted by $C_{sing}(\Phi^+)$ and its complement in $C(\Phi^+)$ is denoted by $C_{reg}(\Phi^+)$.  The cone $C_{reg}(\Phi^+)$ is open and the connected components of  $C_{reg}(\Phi^+)$ are  open conic chambers called \textit{big chambers}. 

Given a big chamber $\mathfrak{c}$, we can define a functional $\chi_{\mathfrak{c}}$ on $S_H$, where for $\sigma \in B(\Phi^+)$ ,
$\chi_{\mathfrak{c}}(f_{\sigma}) := vol(\text{Box}(\sigma))$, if 
$\mathfrak{c} \subset C(\sigma)$, and $\chi_{\mathfrak{c}} = 0$ otherwise. Note that $\text{Box}(\sigma)$ for a basis $\sigma$ will be the parallepiped spanned by the basis vectors.

We can now state the main theorem which relates total residues and power series expansions.

\begin{thm}\cite[Theorem 2.3]{sv}
\label{thm2.1}
   For a function $F \in M^{\Phi^{+}}$ and $\mathfrak{c}$ a big chamber of $C(\Phi^{+})$, Suppose $\lambda \in \Gamma^*$ such that $(\lambda + N(F)^{\circ}) \cap \mathfrak{c} \neq \phi $. For such $\lambda$, 
   \[
    c(\lambda,r^{+}(F)) = \chi_{\mathfrak{c}}(s[F, \Gamma](\lambda)) = \sum_{p \in RSP(F)} \chi_{\mathfrak{c}}(\textrm{Tres}(e^{-\lambda(p-z)}F(p-z))).
   \]

\end{thm}
\begin{example}
    Let us look at $A_2$ root system inside $V^{*} = \{v \text{ | } v \in \CC^*, \text{ } v(e_1+e_2 + e_3) = 0 \}$, with $\Phi^+ = \{\alpha_1 = e_{1} - e_{2},\alpha_2 = e_{2}-e_{3},\beta = \alpha_1 + \alpha_2\}$.  The set of possible bases, $B(\Phi^+) = \{\{\alpha_1,\alpha_2\},\{\alpha_1,\alpha_1 + \alpha_2\},\{\alpha_1 + \alpha_2,\alpha_2\}\}$.  The role of  $\Gamma$  will be played by $Q^{\vee}(A_{2})$ and thus $\Gamma^* = P(A_2)$.  We write $t \in \Hom(P, \CC^*) = e^{z}$ where $z \in V$. We fix a coordinate system $z = (z_1,z_2)$ dual to $\{\alpha_1, \alpha_2\}$. The weight lattice $P$ has a basis of fundamental weights $\omega_1, \omega_2$ dual to the co-roots. We define a volume form on $V^*$, with the unit parallepiped of fundamental weights having volume 1. With respect to this form, the volume of the parallepiped determined by $\alpha_{1}, \alpha_{2}$ will be the determinant of the change of basis matrix $(\langle\alpha_{i},\alpha_{j}^\vee\rangle)$ which will be 3 in this case.  This will also be the cardinality of the quotient lattice $P/Q$.  Consider the following function,
    
\[\eta(z_{1},z_{2}) = \prod_{\alpha \in \Phi^+}(\frac{(1-e^{\alpha(z)})^2}{(1-qe^{\alpha(z)})(1-q^{-1}e^{\alpha(z)})}).\]
    
The singular vectors are the spans of each of the positive roots while the connected components of the regular part are the interiors of $C(\alpha_1,\alpha_1 + \alpha_2)$ and $C(\alpha_2,\alpha_1 + \alpha_2)$. 
    
    If $\mathfrak{c} = C(\alpha_1,\alpha_1 + \alpha_2)^{\circ} $ we have $\chi_{\mathfrak{c}}(\frac{a}{z_1(z_1+z_2)} + \frac{b}{z_2(z_1+z_2)} + \frac{c}{z_1(z_2)}) = a + c.$

    Let $\lambda = a_{1}\alpha_{1} + a_{2}\alpha_{2}$. We will compute compute \[ \chi_{\mathfrak{c}}(s[\eta, \Gamma](\lambda)) = \sum_{p \in RSP(\eta)} \chi_{\mathfrak{c}}(\textrm{Tres}(e^{-\lambda(p-z)}F(p-z))).\]

If  $p = (p_1,p_2)$ is a representation of an element in $RSP(\eta)$,  we should have\[ e^{\alpha(p)}q = 1 \text{ or } e^{\alpha(p)}q^{-1} = 1  \text{ for }  \alpha \in B ,\]where $B \in B(\Phi^+)$. For a fixed choice of B and $\{e^{\alpha(p)}\}_{\alpha \in B}$, we have a unique class of such points $p$ modulo $2\pi i Q^*$, where $Q^*$ is the lattice dual to $Q$. $Q^{\vee} \subset Q^{*}$ and since $RSP$ is the set of points $p$ modulo $2\pi i Q^\vee$ we will have $|Q^*/Q^\vee|$  classes in $RSP$ corresponding to a choice of $B$ and $\{e^{\alpha(p)}\}_{\alpha \in B}$ . We can calculate $|Q^*/Q^\vee|$ taking duals:
    \[|Q^*/Q^\vee| = |(Q^\vee)^*/(Q^*)^*| = |P/Q| = 3. \]

    In each case using  \cref{lem2.1}, we can calculate $\chi_{\mathfrak{c}}(s[\eta, \Gamma](\lambda))$. We have the following cases with non-zero contribution:
    \begin{enumerate}
    \item \[e^{p_1} = q, e^{p_2} = q, e^{p_1 + p_2} = q^2,\]\[\chi_{\mathfrak{c}}(\textrm{Tres}(e^{-\lambda(p-z)}\eta(p-z))) = \frac{1}{3}q^{-a_1-a_2}\frac{(1-q)^{4}}{z_1z_2(1-q)(1-q^3)},\]
    \item \[e^{p_1} = q^{-1}, e^{p_2} = q^{-1}, e^{p_1 + p_2} = q^{-2},\]\[\chi_{\mathfrak{c}}(\textrm{Tres}(e^{-\lambda(p-z)}\eta(p-z))) =  \frac{1}{3}q^{a_1+a_2}\frac{(1-q^{-1})^{4}}{(1-q^{-1})(1-q^{-3})},\]
    
   \item\[
    e^{p_1} = q, e^{p_2} = q^{-2}, e^{p_1 + p_2} = q^{-1}\]\[
    \chi_{\mathfrak{c}}(\textrm{Tres}(e^{-\lambda(p-z)}\eta(p-z))) =  \frac{1}{3}q^{-a_1+2a_2}\frac{(1-q)^2(1-q^{-1})^2(1-q^{-2})^2}{(1-q^2)(1-q^{-2})(1-q^{-1})(1-q^{-3})},
    \]
    \item\[
    e^{p_1} = q^{-1}, e^{p_2} = q^{2}, e^{p_1 + p_2} = q\]\[
    \chi_{\mathfrak{c}}(\textrm{Tres}(e^{-\lambda(p-z)}\eta(p-z))) = \frac{1}{3} q^{a_1-2a_2}\frac{(1-q)^2(1-q^{-1})^2(1-q^2)^2}{(1-q^2)(1-q^{-2})(1-q^{1})(1-q^{3})}.
    \]
    \end{enumerate}

 Note that each case corresponds to three classes with the exact same residue. Thus, on adding the terms,  we get
    \[\chi_{\mathfrak{c}}(s[\eta, \Gamma](\lambda))= q^{-a_1-a_{2}}(q-1)^{-4}\frac{q^{3a_2}-1}{q^3-1}\frac{q^{2a_1-a_2}-1}{q-1}.\]

\end{example}

 \section{Computing the Trace using Residues}
 \label{section4}

Let $\Phi \subset V^*$ be an irreducible root system.  We fix a decomposition $\Phi = \Phi^+ \cup \Phi^{-1}$. We have the weight lattice $P \subset V^*$,  which will play the role of  $\Gamma^*$  from the previous section.  The dual lattice will be the coroot lattice $Q^{\vee}$. We have a natural volume form on $P$ where the parallepiped of fundamental weight (i.e the forms dual to the simple co-roots) has volume 1.\\

 While in $\eta(t)$,  $t$ lies in the torus $\Hom(P,\CC^*)$, we can pull back to the Lie algebra of  $\Hom(P,\CC^*)$, which is isomorphic to $V$ via the exponential map and so we have the natural function that we again denote by $\eta(z)$ whose domain is $V$ given by 
 \[\eta(z) = \Pi_{\alpha \in \Phi^{+}} \frac{(1-e^{\alpha(z)})^2}{(1-q^{-1}e^{\alpha(z)})(1-qe^{\alpha(z)})} = \Pi_{i =1 }^r \frac{1-e^{\alpha_i(z)}}{1-u_ie^{\alpha_i(z)}},\] 
 
 where $\alpha_{1},\alpha_{2},\ldots,\alpha_{r}$ is the collection of all positive roots occurring with multiplicity 2. An element $t \in \Hom(P,\CC^*)$ determines a unique class of points $p$ modulo $2\pi i Q^{\vee}$, such that $e^{p} = t$.

 Restating \cref{thm2.2} in the context of Section \ref{section3}, it states that $c(\lambda,r^+(\eta)) = \tau(\theta_{-\lambda}).$

 We can expand the numerator of $\eta$ to get \[\eta(z) = \frac{\sum_{i_1<i_2 \cdots <{i_k}} \pm 1 e^{(\alpha_{i_1} + \cdots +\alpha_{i_k})(z)}}{\Pi_{i =1 }^r(1-u_ie^{\alpha_i(z)})}.\]

 Note that $RSP(\eta)$, depends only on  $\Pi_{i =1 }^r(1-u_ie^{\alpha_i(p-z)})$
 and  henceforth, we refer to this set as $RSP$.

 Let $\mathfrak{c}$ be a big chamber of the cone $C(\Phi^+)$. Let us fix $\lambda \in \mathfrak{c}$.
 We have the following result:

 \begin{thm}\label{thm3.1}
   For $\lambda \in  \mathfrak{c}  \cap  Q^+$,
   \[
    c(\lambda,r^+(\eta)) = \sum_{t_p \text{ is a residual point}}{\chi_{\mathfrak{c}}(\textrm{Tres}(e^{-\lambda(p-z)}\eta(p-z))}.
   \]
     
 \end{thm}

 \begin{proof}

 We will show that one can apply  \cref{thm2.1} to $\frac{\sum_{i_1<i_2 \cdots <{i_k}} \pm 1 e^{(\alpha_{i_1} + \cdots + \alpha_{i_k})(z)}}{\Pi_{i =1 }^r(1-u_ie^{\alpha_i(z)})}$ .  We have $ 0 \in N(\frac{e^{(\alpha_{i_1} + \cdots \alpha_{i_j})(z)}}{\Pi_{i =1 }^r(1-u_ie^{\alpha_i(z)})}) = (\sum _{i=1}^r [0,1] \alpha_{i})- (\alpha_{i_1} + \cdots \alpha_{i_k})$, since $  \alpha_{i_1} + \cdots \alpha_{i_k} \in \sum _{i=1}^r [0,1] \alpha_{i}$. Furthermore,$(\sum _{i=1}^r [0,1] \alpha_{i})- (\alpha_{i_1} + \cdots + \alpha_{i_k})$ is homeomorphic to a closed disk $D^{n}$ , and thus we have $\lambda \in N(\frac{e^{(\alpha_{i_1} + \cdots \alpha_{i_k})(z)}}{\Pi_{i =1 }^r(1-u_ie^{\alpha_i(z)})}) + \lambda$ and any open neighborhood of $\lambda$ intersects $N(\frac{e^{(\alpha_{i_1} + \cdots \alpha_{i_k})(z)}}{\Pi_{i =1 }^r(1-u_ie^{\alpha_i(z)})})^{\circ}$. In particular,  since $\mathfrak{c}$ is an open set containing $\lambda$ , $\mathfrak{c} \cap N(\frac{e^{(\alpha_{i_1} + \cdots \alpha_{i_k})(z)}}{\Pi_{i =1 }^r(1-u_ie^{\alpha_i(z)})})^{\circ} \neq \phi $.

 Thus, we can apply \cref{thm2.1} to $\frac{e^{(\alpha_{i_1} + \cdots \alpha_{i_k})(z)}}{\Pi_{i =1 }^r(1-u_ie^{\alpha_i(z)})}.$

\begin{equation*}
    \begin{split}
    c(\lambda,r^+(\eta)) &=\sum_{i_1<i_2 \cdots < {i_k}} c(\lambda,r^{+}(\frac{\pm e^{(\alpha_{i_1} + \cdots \alpha_{i_k})(z)}}{\Pi_{i =1 }^r(1-u_ie^{\alpha_i(z)})}))\\
    &= \sum_{i_1<i_2 \cdots < {i_k}} \sum_{p \in RSP} \chi_{\mathfrak{c}}(\textrm{Tres}(e^{-\lambda(p-z)}\frac{\pm e^{(\alpha_{i_1} + \cdots \alpha_{i_k})(p-z)}}{\Pi_{i =1 }^r(1-u_ie^{\alpha_i(p-z)})})) \\
    &= \sum_{p \in RSP} \chi_{\mathfrak{c}}(\textrm{Tres}(e^{-\lambda(p-z)} \frac{ (\sum_{i_1<i_2 \cdots < {i_k}} \pm e^{(\alpha_{i_1} + \cdots \alpha_{i_k})(p-z)}}{\Pi_{i =1 }^r(1-u_ie^{\alpha_i(p-z)})}))\\
    &= \sum_{p \in RSP}{\chi_{\mathfrak{c}}(\textrm{Tres}(e^{-\lambda(p-z)}\eta(p-z))}.
    \end{split}
\end{equation*} 

Suppose $p$ is in $RSP$.  $\prod_{i=1}^r(1-u_ie^{\alpha_{i}(p-z)})$ has a zero of some order $ i_p \geq n$. Additionally, $\prod_{i=1}^r(1-e^{\alpha_{i}(p-z)})$  would vanish at zero with order $z_p$. For $\chi_{\mathfrak{c}}(\textrm{Tres}(e^{-\lambda(p-z)}\eta(p-z))$ to be non zero  we must have $i_p - z_p \geq n$. If we define $t_p\in Hom(\Gamma^*, \CC^*) $ by $t_p(\alpha) = e^{\alpha(p)}$, this means that $t_p$ is a Residual Point in the sense of \cite{op2}. Thus , we can restrict the sum to residual points.
 \end{proof}

\subsection{Type A}
Let us specialize to the case of the root system $A_{n-1} \subset V^*$, where $V^* \subset \CC^n$ is the subspace of all vectors whose coordinates add up to zero, we choose $\Phi^{+}_n = \{\alpha_{i,j} = e_j - e_i | 1 \leq j <i \leq n\}$. The simple roots are $\alpha_{i} = \alpha_{i,i+1}$ for $1 \leq i <n$. Associated to a simple root $\alpha_i$ is the reflection $s_i$ across its normal hyperplane. 

The $s_i$ generate a group $W(A_{n-1})$ isomorphic to $S_n$, acting on the indices of $\CC_n$. We also have negative roots $\Phi^{-}_n = \{\alpha_{i,j} = e_i - e_j | 1 \leq j <i \leq n\}$ and together we have $\Phi_n = \Phi^{+}_n \cup \Phi^{-}_n$.  For a general root $\alpha$, $|\alpha|$ will represent the corresponding positive root. For a permutation $w$,  Let $\epsilon_n(w)$ be the number of simple roots in $\Phi_{n}^{+}$ that are sent to positive roots by $w$. 

\subsection{Calculating $c(\lambda,\eta)$}
For $i,j$ we define \[
\Omega_{i,j} = \Omega(\alpha_{i,j}) = (\frac{(q^{j-i}-1)^2}{(q^{j-i-1}-1)(q^{j-i+1}-1)} )^*,
\]
where $()^*$ means a product of only non-zero factors.
We have\[ \Omega_{i,j} = \Omega_{j,i} \; |j-i|>1,\] \[\Omega_{i,i+1} = - \Omega_{i+1,i}.\]
We also define \[\Omega(A_{n-1}) = \prod_{1 \leq i<j \leq n} \Omega_{ij}.\]

We can use \cref{thm2.3} to describe the residual points in type $A$.

For $A_{n-1}$, the corresponding simply connected Lie group $G$ is $SL_n(\CC)$ and its maximal torus $T$ consists of diagonal matrices. Since there is no proper subgroup of maximal dimension,  if $t = cs$ is a residual point, $C_{\frak{g}}(s) = \frak{sl}_n(\CC)$ and the only possible $s$ are just the points given by the $n$ scalar matrices $e^{r_j}$ in $SL_n(\CC)$, where $r_j = (\frac{2\pi j\sqrt{-1}}{n},\cdots,\frac{2\pi j \sqrt{-1}}{n})$. The Lie algebras $\frak{sl}_{n}(\CC)$ has a single distinguished nilpotent class with the corresponding weighted Dynkin diagram having weights $2$ on each node. Let $h$ be the corresponding element of $Q^\vee$ such that  $\alpha_i(h) = 2,\text{ } \forall 1 \leq i \leq n-1 $ .
Thus, the residual points are given by  $t_{p_w}s_j= e^{r_j+p_w}$ where $w \in W(A_{n-1})$ and $p_w = e^{\frac{k(wh)}{2}}$. The corresponding root system $\Phi_{s,1} = \Phi(\mathfrak{sl}_n) = A_{n-1} = \Phi$  and thus $e^{r_j(\alpha)} = 1$ for every $1 \leq j \leq n$ and $\alpha \in A_{n-1}$. 

\begin{defn}
    For a big chamber $\mathfrak{c}$ of $C(\Phi_{n}^+)$, we define  
    $W(\mathfrak{c}, A_{n-1}) = \{w \in W(A_{n-1}) |  \mathfrak{c} \subset C(|w\alpha_1|,\cdots,|w\alpha_{n-1}|)\}.$
\end{defn}

\begin{prop}
\label{thm4.2}
\
Suppose $\mathfrak{c}$ is a big chamber of $C(\Phi^{+}_n)$.

For  $\lambda \in \mathfrak{c} \cap Q^{+}$,
\[c(\lambda, r^+(\eta)) = \sum_{w \in W(\mathfrak{c}, A_{n-1}) } (-1)^{\epsilon_n(w)}e^{-\lambda(p_{w})}\Omega(A_{n-1}).\]

\end{prop}

\begin{proof}
    To use  \cref{thm3.1} in order to calculate $c(\lambda,r^{+}(\eta))$, we must first calculate the total residue of $\eta$ at the residual points. We first consider residual  points of the form $s_jt_{p_{id}}$,  where $s_j(\alpha) = e^{r_{j}(\alpha)}= 1 $ for all $\alpha$, and $t_{\frac{kh}{2}= p_{id} }=$ $e^{\frac{kh}{2}}$, where $t_{p_{id}}(\alpha_{i}) = q$ for all $i$. We see that  \[e^{-\lambda(r_{j}+p_{id}-z)}F(r_{j}+ p_{id}-z) = e^{-\lambda(p_{id}-z)}F( p_{id}-z),\] and in particular

\[\textrm{Tres}(e^{-\lambda(r_{j}+p_{id}-z)}F(r_{j}+ p_{id}-z)) = \textrm{Tres}(e^{-\lambda(p_{id}-z)}F( p_{id}-z)).\]
Since $\alpha_{i,j} = \sum_{k=i}^{j-1} \alpha_{k}$,  $t_{p_{id}}(\alpha_{ij}) =  q^{j-i}$ and 

 \begin{equation*}
    \begin{split}
    e^{-\lambda(p_{id}-z)}F( p_{id}-z) &= e^{-\lambda( p_{id}-z)}\prod_{1 \leq i < j \leq n} \frac{(1-e^{\alpha_{ij}(p_{id}-z)})^2}{(1-q^{-1}e^{\alpha_{ij}(p_{id}-z)})(1-qe^{\alpha_{ij}( p_{id}-z)})}\\ &= e^{-\lambda(p_{id})}e^{\lambda(z)} \prod_{1 \leq i < j \leq n} \frac{(1-q^{j-i}e^{\alpha_{ij}(-z)})^2}{(1-q^{j-i-1}e^{\alpha_{ij}(-z)})(1-q^{j-i+1}e^{\alpha_{ij}(-z)})}.\\
    \end{split}
\end{equation*} 

The factors of  $e^{-\lambda(p_{id}-z)}F( p_{id}-z)$ that vanish at zero are precisely $(1-q^{i+1-i-1}\alpha_{i,i+1})$ and since the  $\alpha_{i,i+1}$ form a basis, we can apply \cref{lem2.1} and conclude that
\begin{equation}
\label{equation*4.1}
\begin{split}
&\textrm{Tres}(e^{-\lambda(p_{id}-z)}F(p-z))\\
&= \frac{1}{\prod \alpha_{i,i+1}(z)}[e^{-\lambda(p_{id})}e^{\lambda(z)}\prod \alpha_{i,i+1}(z)\prod_{1 \leq i < j \leq n} \frac{(1-q^{j-i}e^{\alpha_{ij}(-z)})^2}{(1-q^{j-i-1}e^{\alpha_{ij}(p-z)})(1-q^{j-i+1}e^{\alpha_{ij}(z)})}]_0 \\ 
&=  \frac{1}{\prod\alpha_{i}(z)} e^{-\lambda(p_{id})}\prod_{1 \leq i<i+1 < j \leq n} \frac{(1-q^{j-i})^2}{(1-q^{j-i-1})(1-q^{j-i+1})}\prod_{1 \leq i < n }\frac{(1-q)^2}{1(1-q^2)} \\ 
&=  \frac{e^{-\lambda(p_{id})}}{\prod\alpha_{i}(z)} \prod_{1 \leq i<i+1 < j \leq n} \frac{(q^{j-i}-1)^2}{(q^{j-i-1}-1)(q^{j-i+1}-1)}\prod_{1 \leq i < n }\frac{(-1)(q-1)^2}{(q^2-1)}. \\
\end{split}
\end{equation}

We can rewrite \cref{equation*4.1} as
\[\textrm{Tres}(e^{-\lambda(p_{id}-z)}F(p-z)) = \frac{(-1)^
{n-1}e^{-\lambda(p_{id})}}{\prod\alpha_{i,i+1}(z)} \Omega(A_{n-1}).\]

A general  point of the form $s_jt_{p_{w}}=s_jwt_{p_{id}}$($n$  such points) satisfies $$s_jt_{p_{w}}(w(\alpha_{i}) = \alpha_{w(i)w(i+1)}) = q. $$ Once again, we will have \[
\textrm{Tres}(e^{-\lambda(r_j + p_{w}-z)}F(r_j + p_{w}-z)) = \textrm{Tres}(e^{-\lambda(p_{w}-z)}F(p_{w}-z)).
\]
 If $w(\alpha_{i}) < 0 $, then in terms of $|\alpha_{w(i),w(i+1)}| \in \Phi^{+}$  we have \[s_jt_{p_{w}}(|\alpha_{w(i),w(i+1)}| = \alpha_{w(i+1),w(i)}) = q^{-1}.\] In general, for $i < j$  \[s_jt_{p_{w}}(\alpha_{ij} = \alpha_{w(w^{-1})(i)w(w^{-1})(j)}) = t_{p_{id}}(\alpha_{w^{-1}(i)w^{-1}(j)}) =  q^{w^{-1}(j)-w^{-1}(i)}.\]
The  positive roots  for which the corresponding denominator factors of   $e^{-\lambda(p_{w}-z)}F(p_w-z))$  vanish at will be of the form $|w\alpha_{i}|$, $1 \leq i\leq n-1$. Similar to $p_{id}$, we have 

\begin{equation*}
\begin{split}
&\textrm{Tres}(e^{-\lambda(p_{w}-z)}F(p_{w}-z))\\
&=  \frac{e^{-\lambda(p_{w})}}{\prod |\alpha_{w(i),w(i+1)}|}\prod_{1\leq i < j \leq n} \Omega_{w^{-1}(i)w^{-1}(j)}\\
&= \frac{e^{-\lambda(p_{w})}}{\prod |\alpha_{w(i),w(i+1)}|}\prod_{1\leq i+1<j\leq n} \Omega_{ij}\prod_{1\leq i< n} (-1)^{sgn(w(i) -w(i+1))}\Omega_{ij}\\ 
&= \frac{(-1)^{\epsilon_n(w)}e^{-\lambda(p_{w})}}{\prod |w\alpha_{i}|}\Omega(A_{n-1}).\\
\end{split}
\end{equation*}

Finally, note that the volume of the parallepiped formed by the simple roots in the weight lattice  is the determinant of the Cartan matrix $(\langle \alpha_{i}^{\vee}, \alpha_{j}\rangle)$ of type $A_{n-1}$. (It is the matrix of the change of basis between the fundamental weights and the roots.) Thus, $vol(\text{Box(}\alpha_1, \ldots, \alpha_{n-1})) = n$.  The image of this parallepiped under any $w \in A_{n}$ will also have (absolute) volume = $n$ since $w$ is orthogonal and has determinant $\pm 1$.

Putting everything together:
\begin{equation*}
    \begin{split}
    c(\lambda,r^+(\eta)) 
    &= \sum_{t_p \text{ is a residual point}}{\chi_{\mathfrak{c}}(\textrm{Tres}(e^{-\lambda(p-z)}\eta(p-z))}\\
        &= \sum_{w \in A_{n-1}}\sum_{j=1}^n {\textrm{Tres}(e^{-\lambda(r_j + p_{w}-z)}F(r_j + p_{w}-z))}\\
    &= \sum_{w \in A_{n-1}}\sum_{k=1}^n {\chi_{\mathfrak{c}}(\frac{(-1)^{\epsilon_n(w)}e^{-\lambda(p_{w})}}{\prod |w\alpha_{i}|(z)}\prod_{1\leq i<j\leq n} \Omega_{ij})}\\
    &=  \sum_{w \in W(\mathfrak{c}, A_{n-1})}(-1)^{\epsilon_n(w)}e^{-\lambda(p_{w})}\Omega(A_{n-1}).\\
    \end{split}
\end{equation*} .

\end{proof}

\begin{lem}
        $\Omega(A_{n-1}) = \Omega_{n-1}(A_{n-2})(\frac{(q-1)(q^{n-1}-1)}{(q^{n}-1)}) = \frac{(q-1)^{n}}{(q^{n}-1)}.$
    \end{lem}

    \begin{proof}
        \begin{equation*}
            \begin{split}
                \Omega(A_{n-1}) &= \prod_{1\leq i<j\leq n} \Omega_{ij}\\
                &= \Omega(A_{n-2})(\prod_{1\leq i < n-1}\frac{(q^{n-i}-1)^{2}}{(q^{n-i+1}-1)(q^{n-i-1}-1)})(\frac{(q-1)^2}{(q^2-1)}\\
                &=\Omega(A_{n-2})\frac{(q-1)(q^{n-1}-1)}{(q^n-1)}.\\
            \end{split}
        \end{equation*}
        Since $\Omega(A_{1}) = \frac{(q-1)^2}{q^2-1}$,  by induction, $\Omega(A_{n-1} ) = \frac{(q-1)^n}{q^n-1}$.
    \end{proof}

\begin{defn}
    For $1 < m < n-1 $, 
    We define 
    \begin{equation*}
        \begin{split}
            \mathfrak{a}_{n-1}^m &=  C(\alpha_{1,2}, \ldots, \alpha_{1,n-m},  \alpha_{n-1,n}, \ldots , \alpha_{n-(m-1), n},\alpha_{1,n} )\\
             %C(\beta_1 = \alpha_{1,2}, \cdots, \beta_{n-m-1} = \alpha_{1,n-m}, \beta_{n- m} = \alpha_{n-1,n}, \cdots , \beta_{n-2} = \alpha_{n-(m-1), n}, \beta_{n-1} = \alpha_{1,n} )  \\
            &= \{  \lambda = \sum_{i=1}^{n-1} a_{i} \alpha_{i,i+1} | a_{1} \geq a_{2} \cdots \geq a_{n-m} \leq a_{n-m +1 } \cdots \leq  a_{n-1}  \}.
        \end{split}
    \end{equation*}

Additionally,
    \begin{equation*}
    \begin{split}
         \mathfrak{a}_{n-1}^{1} &= C(\beta_{1} = \alpha_{1, 2}; \cdots , \beta_{n-2} = \alpha_{1, n-1}; \beta_{n-1} = \alpha_{1,n}) \\
        &= \{  \lambda = \sum_{i=1}^{n-1} a_{i} \alpha_{i,i+1} | a_{1} \geq a_{2} \cdots  \cdots \geq  a_{n-1}  \},\\
    \end{split}
    \end{equation*} 
and 
\begin{equation*}
\begin{split}
\mathfrak{a}_{n-1}^{n-1} &= C(\beta_{1} = \alpha_{n-1, n}, \cdots , \beta_{n-2} = \alpha_{2, n}, \beta_{n-1} = \alpha_{1,n}) \\
        &= \{  \lambda = \sum_{i=1}^{n-1} a_{i} \alpha_{i,i+1} | a_{1} \leq a_{2} \cdots  \cdots \leq  a_{n-1}  \}.\\
\end{split}
\end{equation*}

\end{defn}
\begin{comment}
    \textbf{Remark}: 
   An element $c \in A_{n-1}$ is a Coxeter element if \[c = \prod_{k=1}^n s_{i_k}\] where the $s_{i_1},\cdots,s_{i_{n-1}}$ are $s_{1},\cdots,s_{n-1}$ in some order.
 To each Coxeter element $c = \prod_{k=1}^n s_{i_k} $ we can associate positive roots $\beta_1 = \alpha_{i_1}, \beta_2 = s_{i_1}(\alpha_{i_2}), \cdots, \beta_{n-1} = s_{i_1}s_{i_2}\ldots s_{i_{n-2}}(\alpha_{i_{n-1}})$
 Our chambers $\mathfrak{a}_{n-1}^{m}$ are spanned by the roots associated to Coxeter elements of the form. $c = s_1\cdots s_{n-m-1}s_{n-1}\cdots s_{n-m}$.
\end{comment}

\begin{prop}
    The interior of $\mathfrak{a}_{n-1}^m$ is a big chamber for all $0 \leq 1 \leq m$.
\end{prop}

\begin{proof}
    A vector in the interior of $\mathfrak{a}_{n-1}^m$ is of the form $$v = r_{1}\alpha_{1,2}+ \cdots r_{n-m-1}\alpha_{1, n-m} + a_{n-m} \alpha_{1,n} + r_{n-m+1}\alpha_{n-m+1,n} \cdots  r_{n-1}\alpha_{n-1,n},
    $$ where all the coefficients are non-zero.  The vector $v$ will lie inside a big chamber if and only if $v$ cannot be written as a positive linear combination of less than $n-1$ positive roots.

    Consider a general expansion of v  a positive linear combination of a collection $P$ of positive roots $\sum_{P} a_{\alpha} \alpha = v$. Note that in the standard basis $e_1,\ldots, e_n$, $v$ has negative coordinates in $e_2,\ldots, e_{n-m}$. This implies the existence of a positive root $\alpha_{i,j}$ for each $2 \leq j \leq n-m$ in $P$. \\
    
  Similarly, $v$ has positive coordinates in $e_{n-m+1},\cdots,e_{n-1}$. Thus, we must have positive roots $\alpha_{i,j}$ in $P$ for each $n-m+1 \leq i \leq n-1$.\\
    
Thus, $|P| \geq n-m-1 + m-1 = n-2$. If $|P| = n-2$,  $P \subset C(\Delta \setminus \alpha_{n-m,n-m+1})$ but as a combination of simple roots  $v$ will be supported on $\alpha_{n-m,n-m+1}$. Hence,  $|P| \geq n-1$. 

The boundary of $a_{n-1}^m$ consists of singular vectors and so $(a_{n-1}^m)^{\circ}$ is separated from any  regular vector outside  $(a_{n-1}^m)^{\circ}$, making it a connected component. 
    Thus $(\mathfrak{a}_{n-1}^m)^o$  must be a big chamber.
\end{proof}

Note that we have an involution $-w_{n-1}$ on the positive roots of $\Phi_{n}^{+}$ that sends $\alpha_{i,j}$ to $\alpha_{n+1-j,n+1-i}$. This extends to an involution on $C(\Phi_{n-1}^+)$ that sends $\mathfrak{a}_{n-1}^m$ to $\mathfrak{a}_{n-1}^{n-m}$.
We will denote the sets of permutations $W(\mathfrak{a}_{n-1}^m, A_{n-1})$ by $W^m_{n-1}$.

\subsubsection{Trace formula for $\mathfrak{a}_{n-1}^{1}$}
We will first establish \cref{thm1.1} for $\lambda \in \mathfrak({a}_{n-1}^1)^{\circ} \cap Q^+$
\begin{prop}
\label{prop4.4}

 Let $H$ be the extended affine Hecke algebra of type $A_{n-1}$. If \[\lambda = \sum_{i=1}^{n-1} a_i \alpha_{i} \in (\mathfrak{a}_{n-1}^1)^{\circ} \cap Q^+,\]
then \[\tau(\theta_{-\lambda}) = q^{-(\sum_{i=1}^{n-1}a_i)}(q-1)^{2(n-1)}[a_{n-1}]_{q^n} \Pi_{1 \leq i < n} [(i+1)a_i - ia_{i+1}]_{q}. \]
    
\end{prop}

In order to apply \cref{thm4.2}  we must first determine $ W^1_{n-1}$.

 Let $w \in W^1_{n-1} $ , for which we have 
$w\Delta = \{\alpha_{w(i)w(i+1)}\}$. If we look at $ww_{n-1}$, where $w_{n-1}$ is the longest element we know $w_{n-1}(i) = n+1-i$. Thus $ww_{n-1}\Delta = \{\alpha_{w(n+1-i)w(n+1-(i+1))}\}$, which means that $ww_{n-1}\Delta = -w\Delta$ and in particular $\{|\alpha_{w(i)w(i+1)}|\} = \{|\alpha_{w(n+1-i)w(n+1-(i-1))}|\} $. Thus, we can conclude that $ww_{n-1} \in W^1_{n-1}$ .

\begin{lem}\label{lem3.2}
        \[W^1_{n-1} =  W^1_{n-2} \cup W^1_{n-2}  w_{n-1},\]where $W^1_{n-2}$ acts as a permutation of $1,\cdots,n-1$, leaving $n$ fixed
    \end{lem}
    \begin{proof}
    Let $w \in W^{1}_{n-1}$. We observe that 
    \[\mathfrak({a}^{1}_{n-1})^{\circ} \subset C(|\alpha_{w(1),w(2)}|,\ldots,|\alpha_{w(n-1),w(n-2)}|) \iff \alpha_{1,r} \in C(|\alpha_{w(1), w(2)}|,\cdots,|\alpha_{w(n-2)w(n-1)}|), \] for all $1 \leq r \leq n-1$.
    
We claim that $w(n) = n \text{ or } 1$. If $1<w^{-1}(n) = i < n$,  any positive linear combination with the  root $\alpha_{w(w^{-1}(n)-1)n}$ or $ \alpha_{w(w^{-1}(n)+1)n}$ contains an $e_{n}$ component.  $w(\Delta)\setminus \{\alpha_{w(w^{-1}(n)-1)n},  \alpha_{w(w^{-1}(n)-1)n}\}$ is a set of size $n-3$ and it spans a  cone that contains $n-2$ linearly independent elements $\alpha_{1,2} ,\cdots \alpha_{1,n-2}$ which is not possible. So we have $w(n) = n$ or $w(1) = n$. 

If $w(n) = n$ ,  $w \in W_{n-2} \subset W_{n-1}$ and $C(w(\Delta_{n-1})\setminus \{\alpha_{w(n-1)n}\}) = C(w(\Delta_{n-2}))$ must contain $\mathfrak{a}_{n-2}^1$. Thus $w \in W_{n-2}^{1}$. \\

    If $w(1) = n$, we have $ww_{n-1}(n) = n$ and $ww_{n-1} \in W^{1}_n$ and so $ww_{n-1} \in W^1_{n-1}$.
\end{proof}

\begin{lem}
\label{lem4.3}
 Let   $\lambda = \sum^{n}_{i=1} a_{i}\alpha_{i} \in \mathfrak({a}^{1}_{n})^{\circ}$ and $\lambda'  = \sum^{n-1}_{i=1} a_{i}\alpha_{i} \in (\mathfrak{a}^{1}_{n-1})^{\circ} $. For $w \in W^{1}_{n-2}$,
  \[
    (-1)^{\epsilon_{n+1}(w)}e^{-\lambda(p_{w})} = (-1)^{\epsilon_n(w)+1}e^{-\lambda'(p_{w})}q^{- a_n}, \]
    \[ (-1)^{\epsilon_{n+1}(ww_{n})}e^{-\lambda(p_{ww_{n}})} = (-1)^{\epsilon_n(ww_{n-1})}e^{-\lambda'(p_{ww_{n-1}})}q^{a_n},\\
    \]
    \[
    (-1)^{\epsilon_{n+1}(ww_{n-1})}e^{-\lambda(p_{ww_{n-1}})} = (-1)^{\epsilon_{n}(ww_{n-1})+1}e^{-\lambda'(p_{ww_{n-1}})}q^{-na_n},
    \]
    \[
    (-1)^{\epsilon_{n+1}(ww_{n-1}w_{n})}e^{-\lambda(p_{ww_{n-1}w_{n}})} = (-1)^{\epsilon_{n}(ww_{n-1}w_{n-1})}e^{-\lambda'(p_{ww_{n-1}w_{n-1}})}q^{na_n}.
    \]
 
\end{lem}

\begin{proof}
Using \cref{lem3.2} , to each element in $w \in A_{n-2}^{1}$ we can associate 2 elements $w$ and $ww_{n-1}$ in $A_{n-1}^{1}$ and 4 elements $w$, $ww_{n}$, $ww_{n-1}$,  $ww_{n-1}w_{n}$ in $A_{n}^{1}$. For each element we have \[
    \epsilon_{n+1}(w) = \epsilon_{n}(w)+1 \text{, }e^{\alpha_{n,n+1}(p_{w})} = q, \]
   \[
    \epsilon_{n+1}(ww_{n}) = \epsilon_{n}(ww_{n-1})  \text{, }e^{\alpha_{n,n+1}(p_{ww{n}})} = q^{-1}, \]
    \[
    \epsilon_{n+1}(ww_{n-1}) = \epsilon_{n}(ww_{n-1})+1 \text{, }e^{\alpha_{n,n+1}(p_{ww_{n-1}})} = q^{n},
    \]
    \[
    \epsilon_{n+1}(ww_{n-1}w_n) = \epsilon_{n}(ww_{n-1}w_{n-1})  \text{, }e^{\alpha_{n,n+1}(p_{ww_{n-1}w_n})} = q^{-n}.
    \]The equalities follow from direct calculation. \mycomment{The corresponding terms in $\sum_{w \in W^{1}_{n}}{(-1)^{\epsilon(w)}}e^{-\lambda(p_{w})}$ are,
     
    \[
    (-1)^{\epsilon(w)}e^{-\lambda(p_{w})} = (-1)^{\epsilon'(w)+1}e^{-\lambda'(p_{w})}q^{- a_n}, \]
    \[ (-1)^{\epsilon(ww_{n})}e^{-\lambda(p_{ww_{n-1}})} = (-1)^{\epsilon'(ww_{n})}e^{-\lambda'(p_{ww{n}})}q^{a_n},\\
    \]
    \[
    (-1)^{\epsilon(ww_{n-1})}e^{-\lambda(p_{ww_{n-1}})} = (-1)^{\epsilon'(ww_{n-1})+1}e^{-\lambda'(p_{ww_{n-1}})}q^{-na_n},
    \]
    \[
    (-1)^{\epsilon(ww_{n-1}w_{n})}e^{-\lambda(p_{ww_{n-1}w_{n}})} = (-1)^{\epsilon'(ww_{n-1}w_{n-1})}e^{-\lambda'(p_{ww{n-1}w{n-1}})}q^{na_n},
    \]

    These $4$ terms make the following contribution to the sum.
    \[((-1)^{\epsilon'(w)}e^{-\lambda'(p_{w})}q^{-a_{n}}(q^{(n+1)a_n}-1) + (-1)^{\epsilon'(ww_{n-1})}e^{\lambda'(ww_{n})}q^{-na_{n}}(q^{(n+1)a_n}-1) \]
    \[
    = q^{-a_n}(q^{(n+1)a_n}-1)((-1)^{\epsilon'(w)}e^{-\lambda'(p_{w})} + (-1)^{\epsilon'(ww_{n})}e^{\lambda'(ww_{n})}q^{-(n-1)a_{n}})
    \]

    Summing over $w \in W^1_{n-2}$ we have the desired equality
}
\end{proof}

We can now prove \cref{prop4.4}.

\begin{proof}
    We will proceed via induction. The cases $n=2,3$ was verified by direct computation in Section 1 and Section 2, in combination with \cref{thm3.1}.
   
Suppose the theorem is true for $n$. We will show that the theorem is for $n+1$.
    Let  $\lambda = \sum^{n}_{i=1} a_{i}\beta_{i}$,
    $\lambda' = \sum^{n-1}_{i=1} a_{i}\beta_{i}$. Let $\lambda'' = \sum^{n-2}_{i=1} a_{i}\beta_{i}$.

    \begin{equation*}
    \begin{split}
    &c(\lambda,r^+(\eta))\\
    &=\sum_{w \in W^{1}_{n}}{(-1)^{\epsilon_{n+1}(w)}}e^{-\lambda(p_{w})}\Omega_{n+1}\\
    &=  \sum_{w \in W^{1}_{n-2}} q^{-a_n}(q^{(n+1)a_n}-1)((-1)^{\epsilon_{n}(w)}e^{-\lambda'(p_{w})} + (-1)^{\epsilon_{n}(ww_{n})}e^{\lambda'(p_{ww_{n-1}})}q^{-(n-1)a_{n}})\Omega_{n+1}, \\
    %&=  q^{-a_{n}}[a_n]_{q^{n+1}}(q^{na_{n-1}-(n-1)a_{n}}-1)\sum_{u\in A^{*}_{n-2}} q^{-a_{n-1}}((-1)^{\epsilon_{n-1}(u)}e^{-\lambda''(p_{u})} + (-1)^{\epsilon_{n-1}(uw_{n-2})}e^{\lambda''(p_{uw_{n-1}})}q^{-(n-2)a_{n-2}})\Omega_{n}\\
    &= q^{-a_n}[a_n]_{q^{n+1}}[na_{n-1}-(n-1)a_{n}]_q\frac{(q-1)^2(q^{n}-1)}{(q^{na_{n-1}}-1)}(\sum_{w \in W^{1}_{n-1}}{(-1)^{\epsilon_{n}(w)}}e^{-\lambda''(p_{w})}\Omega_{n}),
    \\
    %&= q^{-a_n}[a_n]_{q^{n+1}}[na_{n-1}-(n-1)a_{n}]_q\frac{(q-1)^2}{[a_{n-1}]_{q^n}}q^{-(\sum_{i=1}^{n-1}a_i)}(q-1)^{2(n-1)}[a_{n-1}]_{q^n} \Pi_{1 \leq i < n} [(i+1)a_i - ia_{i+1}]_{q}\\
   &= q^{-(\sum_{i=1}^{n}a_i)}(q-1)^{2n}[a_{n}]_{q^{n+1}} \Pi_{1 \leq i < n+1} [(i+1)a_i - ia_{i+1}]_{q}.\\
\end{split}
\end{equation*}
\end{proof}

\subsubsection{Trace formula for $\mathfrak{a}_{n-1}^{n-1}$}

\begin{prop}
\label{prop4.5}

 Let $H$ be the extended affine Hecke algebra of type $A_{n-1}$. If $\lambda = \sum_{i=1}^{n-1} a_i \alpha_{i} \in (\mathfrak{a}_{n-1}^{n-1})^{\circ}$,
    then \[\tau(\theta_{-\lambda}) = q^{-(\sum_{i=1}^{n-1}a_i)}(q-1)^{2(n-1)}[a_{1}]_{q^n} \Pi_{1 \leq i < n-1} [(i+1)a_{n-i} - ia_{n-(i+1)}]_{q}. \]
    
\end{prop}

Using \cref{thm4.2}, for  $\lambda \in \mathfrak{a}_{n-1}^{n-1}$ we can write \[ c(\lambda, r^{+}(\eta)) =  \sum_{w \in W^{n-1}_{n-1}}(-1)^{\epsilon_n(w)}e^{-\lambda(p_{w})}\Omega_{n}.\]

 If we embed $W(A_{n-2})$ inside $W(A_{n-1})$ as permutations of $2,\cdots,n$, we can inductively describe $W_{n-1}^{n-1}$.

 \begin{lem}
 \label{lem4.4}
   \[ W_{n-1}^{n-1} = W_{n-2}^{n-2}  \cup W_{n-2}^{n-2}w_{n-1}.\]
 \end{lem}
  
 \begin{proof}
 \begin{equation*}
\begin{split}
     w \in W_{n-1}^{n-1}  &\iff \mathfrak{a}_{n-1}^n \subset C(|\alpha_{w(1),w(2)}|,\ldots|\alpha_{w(n-1),w(n)}|)\\
 &\iff \mathfrak{a}^{1}_{n-1} \subset C(|\alpha_{w_{n}w(1),w_{n}w(2)}|,\ldots |\alpha_{w_{n-1}w(n-1),w_{n-1}w(n)} |)\\
 & \iff w_{n-1}w \in W_{n-2}^1 \cup W_{n-2}^1w_{n-1}\\
 & \iff w \in w_{n-1}W_{n-2}^1 \cup w_{n-1}W_{n-2}^1w_{n-1}.
 \end{split}
 \end{equation*} 
 Conjugation by $w_{n-1}$ maps permutations of $1, \cdots ,n-1$ to permutations of $n, \cdots ,2$ . Thus, we find that \[w_{n-1}W_{n-2}^1w_{n-1} = W_{n-2}^{n-2}.\] Hence,
 \[A_{n-1}^{n-1} = A_{n-2}^{n-2}w_{n-1} \cup A_{n-2}^{n-2}. \]
 \end{proof}.

 We can prove \cref{prop4.5}.
 
 \begin{proof}
 We can write, \[\eta(t) = \Pi_{\alpha \in \Phi^{+}} \frac{(1-t(\alpha))^2}{(1-q^{-1}t(\alpha))(1-qt(\alpha))} = \Pi_{\alpha \in \Phi^{+}} \frac{(1-t(-w_{n-1}(\alpha))^2}{(1-q^{-1}t(-w_{n-1}(\alpha))(1-qt(-w_{n-1}(\alpha))}. \]
 
 Comparing coefficients, we get $c(-w_{n-1}\lambda, r^+(\eta)) = c(\lambda, r^{+}(\eta))$.

 If $\lambda \in \mathfrak{a}_{n-1}^{n-1}$,  $-w_{n-1}\lambda \in \mathfrak{a}_{n-1}^{1}$.  Hence, the formula follows from 
 \cref{prop4.4}.
     
 \end{proof}

 We will need an analogue of \cref{lem4.3} to prove the general case.

\begin{lem}
\label{lem4.5}
 Let   $\lambda = \sum^{n}_{i=1} a_{i}\alpha_{i} \in (\mathfrak{a}^n_{n})^{\circ}$ and $\lambda'  = \sum^{n}_{i=2} a_{i}\alpha_{i} \in (\mathfrak{a}^{n-1}_{n-1})^{\circ} $ Then 
 \[
    (-1)^{\epsilon_{n+1}(w)}e^{-\lambda(p_{w})} = (-1)^{\epsilon_n(w)+1}e^{-\lambda'(p_{w})}q^{- a_1}, \]
    \[ (-1)^{\epsilon_{n+1}(ww_{n})}e^{-\lambda(p_{ww_{n}})} = (-1)^{\epsilon_{n}(ww_{n-1})}e^{-\lambda'(p_{ww_{n-1}})}q^{a_1},\\
    \]
    \[
    (-1)^{\epsilon_{n+1}(ww_{n-1})}e^{-\lambda(p_{ww_{n-1}})} = (-1)^{\epsilon_{n}(ww_{n-1})+1}e^{-\lambda'(p_{ww_{n-1}})}q^{-na_1},
    \]
    \[
    (-1)^{\epsilon(ww_{n-1}w_{n})}e^{-\lambda(p_{ww_{n-1}w_{n}})} = (-1)^{\epsilon'(ww_{n-1}w_{n-1})}e^{-\lambda'(p_{ww_{n-1}w_{n-1}})}q^{na_1}.
    \]
\end{lem}
 
\begin{proof}
     Given $w \in W_{n-2}^{n-2}$, we can associate $w, ww_{n-1} \in W_{n-1}^{n-1}$ and \[w, ww_{n}, ww_{n-1}, ww_{n-1}w_{n} \in W_{n}^{n}.\] Analogous to \cref{lem4.3}, the lemma follows on determining  the value of $\epsilon_{n+1}$ and $e^{\alpha_1}$ for the elements.
      \[
    \epsilon_{n+1}(w) = \epsilon_{n}(w)+1 \text{, }e^{\alpha_{1,2}(p_{w})} = q, \]
   \[
    \epsilon_{n+1}(ww_{n}) = \epsilon_{n}(ww_{n-1})  \text{, }e^{\alpha_{1,2}(p_{ww{n}})} = q^{-1}, \]
    \[
    \epsilon_{n+1}(ww_{n-1}) = \epsilon_{n}(ww_{n-1})+1 \text{, }e^{\alpha_{1,2}(p_{ww_{n-1}})} = q^{n},
    \]
    \[
    \epsilon_{n+1}(ww_{n-1}w_n) = \epsilon_{n}(ww_{n-1}w_{n-1})  \text{, }e^{\alpha_{1,2}(p_{ww_{n-1}w_n})} = q^{-n}.
    \]

\end{proof}

 \subsubsection{Trace formula for $\mathfrak{a}_{n-1}^m$}

We can now establish \cref{thm1.1} in full generality.  We will think of an element in $W(A_{n-m-1}) \times W( A_{m-1})$ as a pair of permutations acting on $1, \cdots, n-m$ and $n-m+1, \cdots ,n$ respectively. 
 $\mathfrak{a}_{n-1}^m =\mathfrak{a}_{n-m-1}^1 +  \mathfrak{a}_{m-1}^{m-1} +\RR^+ \alpha_{1,n} $

 \begin{lem}
 \label{lem4.5}
     \[ W_{n-1}^{m} = W_{n-m-1}^1 \times  W_{m-1}^{m-1} \cup (W_{n-m-1}^1 \times  W_{m-1}^{m-1})w_{n-1}\]
 \end{lem}
 \begin{proof}
     The partition $1,\cdots,n-m$ and $n-m+1, \cdots ,n$ determines a two-colouring of $1,\cdots,n$.
     Suppose $w \in W_{n-1}^{m}$. 
     If  we can find $i-1, i, i+1$ for which $w(i-1),w(i+1)$ have the same colour but $w(i)$ has a different colour, $\alpha_{n-m,n-m+1}$ will be in the support of $\alpha_{|w(i-1),w(i)|}$ and $\alpha_{|w(i),w(i+1)|}$.  We must then have \[\mathfrak{a}_{n-m-1}^1 +  \mathfrak{a}_{m-1}^{m-1} \subset C(|w\alpha_{1}| \cdots |w\alpha_{i-2}|,|w\alpha_{i+1}| \cdots|w\alpha_{n-1}|).  \] But the span of $\mathfrak{a}_{n-m-1}^1 +  \mathfrak{a}_{m-1}^{m-1} $ is a space of dimension $n-2$, while the cone on the right spans a space of dimension $n-3$. Thus the colour of $w(1)$ will determine the entire colouring.
     If $1 \leq w(1) \leq n-m$ ,  \[1 \leq w(1),\cdots,w(n-m) \leq n-m \text{ and }n-m +1 \leq  w(n-m+1),\cdots,w(n) \leq n \]. We must have $w \in  W_{n-m-1}^1 \times  W_{m-1}^{m-1}$. 
     
     Else  $n-m+1 \leq w(1),\cdots,w(m) \leq n \text{ and }1 \leq  w(m+1),\cdots,w(n) \leq n-m $, giving us $ww_{n-1} \in W_{n-m-1}^1 \times  W_{m-1}^{m-1} $.
\end{proof}

We can now prove \cref{thm1.1} in full generality

\begin{proof} 

Using \cref{lem3.2}, to each element in $w \in W_{n-m-2}^{1}$ we can associate 2 elements $w$ and $ww_{n-m-1}$ in $W_{n-m-1}^{1}$.
Similarly, using \cref{lem4.3} each element in $w' \in W_{m-2}^{m-2}$ we can associate 2 elements $w'$ and $w'w_{m-1}$ in $W_{m-1}^{m-1}$.  Using the 4 elements in  $W^{1}_{n-m-2}\times W_{m-2}^{m-2}$ and \cref{lem4.4}, we get 8 elements of $W_{n-1}^m$ corresponding to $(w,w')$.

Summing over $(w,w') \in W^{1}_{n-m-2}\times W_{m-2}^{m-2} $,  we have
    \begin{equation*}
    \begin{split}    
        & \sum_{w \in W^{m}_{n-1}}{(-1)^{\epsilon(w)}}e^{-\lambda(p_{w})} \notag \\
    &= q^{-a_{n-m}}(q^{na_{n-m}}-1)\\&(\sum_{w \in W^{1}_{n-m-2}} ((-1)^{\epsilon'(w)}e^{-\lambda'(p_{w})} + (-1)^{\epsilon'(ww_{n-m-1})}e^{\lambda'(ww_{n-m-1})}q^{-(n-m-2)a_{n-m-1}}) )\\&(\sum_{w' \in W^{m-2}_{m-2}}((-1)^{\epsilon''(w')}e^{-\lambda'(p_{w})} + (-1)^{\epsilon''(w'w_{m-1})}e^{\lambda'(ww_{n-m-1})}q^{-(m-2)a_{n-m-1}}) )\\
    &=q^{-\sum_{i=1}^{n-1} a_{i}}(q^{na_{n-m}}-1)\prod_{i=1}^{n-m-1}(q^{(i+1)a_{i} - ia_{i+1}}-1)\prod_{j=1}^{m-1}(q^{(j+1)a_{n-j}-ja_{n-j-1}}-1).
    \end{split}
    \end{equation*}

\end{proof}

\section{Concluding Remarks}
\label{section5}
 In this work, we are able to prove \cref{thm1.1}, through direct analysis of  the generating function of the trace. The special cases discussed  in \cite{wil}, crucially used calculations from \cite{Ha}, using Tesler matrices, to find the trace of special elements. It would be interesting to see what our results mean for the combinatorics of Tesler matrices. In the language of Tesler matrices, our results correspond to sums of matrices with possibly negative hook sums. In \cite{AGH+}, the emphasis on positive hook sums follows from the need for a minimum number of non zero entries in a Tesler matrix. In our setting, if the corresponding weight is in a big chamber of $C(A_{n-1}^{+})$, it cannot be written as a sum of less than $n-1$ roots. Thus, fixing our weight to be a regular element of the cone $C(A_{n-1}^{+}) $ is a more general condition that ensures the resulting weighted sums are still interesting. 
 
 For general affine Hecke algebras,\cref{thm3.1} reduces calculating the trace to choosing a big chamber and calculating the residue at each residual point.\cref{thm2.3} describes the residue points in a general combinatorial manner. On the other hand, we do not know a general description of the big chambers of the cone generated by positive roots, even in type $A$. However, we can find chambers similar to the ones considered in this paper in other classical types.

The use of residual points in this paper hints at deeper connections with \cite{op2}. The main result in \cite{op2} is a Plancherel formula for the trace, where  the trace is decomposed into a sum of integrals over more general ``residual cosets''. Using the Plancherel theorem
might help in finding the trace of elements beyond the translations.

\bibliographystyle{plain}

\begin{thebibliography}{11}


\bibitem[AGH+]{AGH+}
Drew Armstrong, Adriano Garsia, James Haglund, Brendon Rhoades, and Bruce Sagan, \textit{Combinatorics of Tesler matrices in the theory of parking functions and diagonal harmonics}, Journal of Combinatorics 3 (2012), no. 3, 451–494

\bibitem[EHKM+]{EHKM+}Esther Banaian, Anh Trong Nam Hoang, Elizabeth Kelley, Weston Miller, Jason Stack, Carolyn Stephen, Nathan Williams
\bibitem[BV]{bv}
M.Brion, M.Vergne. \textit{Arrangements of hyperplanes I :Rational Funtions and Jeffery-Kirwan residue}, Annales scientifiques de l'École Normale Supérieure, Serie 4, Volume 32 (1999) no. 5, pp. 715-741

\bibitem[GL]{GL}
Pavel Galashin, Thomal Lam. \textit{Positroids, knots, and -Catalan numbers}  Duke Mathematical Journal 173 (11), 2117-2195

\bibitem[Car]{car}
R. W. Carter. \textit{Finite groups of Lie type} Wiley Classics Library, Wiley, 1993

\bibitem[GLMW]{GLMW}
Pavel Galashin, Thomas Lam, Minh-Tâm Quang Trinh and Nathan Williams.
\textit{Rational Noncrossing Coxeter-Catalan Combinatorics}, arXiv preprint arXiv:2208.00121

\bibitem[Hag]{Ha}
James Haglund. \textit{A polynomial expression for the Hilbert series of the quotient ring of diagonal coinvariants}, Advances in Mathematics 227 (2011), no. 5, 2092–2106

\bibitem[HO]{HO}
 Gert Heckman, Eric Opdam (1997). \textit{Yang’s System of Particles and Hecke Algebras}. Annals of Mathematics, 145(1), 139–173. 
\bibitem[Mac]{Mac}
I.G. Macdonald, \textit{Spherical functions on a group of p-adic type}, Publ. Ramanujan Institute 2 (1971)

\bibitem[LMM]{LMM}
Ricky Ini Liu, Alejandro H. Morales, and Karola Mészáros. \textit{Flow polytopes and the space of diagonal harmonics}, Canadian Journal of Mathematics 71.6 (2019): 1495-1521.

\bibitem[MMR]{MMR}
Karola Mészáros, Alejandro H. Morales, and Brendon Rhoades \textit{The polytope of Tesler matrices} Selecta Mathematica 23.1 (2017): 425-454.

\bibitem[Op1]{op1}
Eric Opdam,
\textit{A generating function for the trace of the Iwahori-Hecke algebra}  Joseph, A., Melnikov, A., Rentschler, R. (eds) Studies in Memory of Issai Schur. Progress in Mathematics, vol 210 (2003).

\bibitem[Op2]{op2}
Eric Opdam,
\textit{On the spectral decomposition of affine Hecke algebras}, Journal of the Institute of Mathematics of Jussieu. 2004;3(4):531–648

\bibitem[Sol]{sol}
Maarten Solleveld, 
\textit{Affine Hecke algebras and their representations},Indagationes Mathematicae
Volume 32 Issue 5 (2021)

\bibitem[SV]{sv}
András Szenes and Michèle Vergne,
\textit{Residue formulae for vector partitions
and Euler–MacLaurin sums},
Advances in Applied Mathematics 30 (2003)

\bibitem[Siz]{siz}
Zsolt Szilágyi
\textit{Computation of Jeffrey–Kirwan residues using Gröbner bases},
Journal of Symbolic Computation
Volume 79, Part 2, March–April (2017)

\bibitem[Wil]{wil}
Nathan Williams,
\textit{Combinatorics of braid varieties}, Proceedings of Symposia in Pure Mathematics (2023+)
\end{thebibliography}
\def\noopsort#1{}
%\begin{thebibliography}{11}

\end{document}